\documentclass[12pt,a4paper]{amsart}
\usepackage{mathrsfs}
\usepackage{amssymb,amsmath,amsthm}
\usepackage{enumerate}
\usepackage{tikz}

\newtheorem{thm}{Theorem}[section]

\newtheorem{lem}[thm]{Lemma}
\newtheorem{defn}[thm]{Definition}
\theoremstyle{definition}

\theoremstyle{definition}

\def\X{\mathcal X}
\def\O{\mathcal O}
\def\A{\mathcal A}
\def\B{\mathcal B}
\def\D{\mathcal D}
\def\C{\mathscr C}
\def\Z{\mathbb Z}
\def\R{\mathbb R}

\def\V{\mathscr V}
\def\x{\widetilde x}
\def\y{\widetilde y}

\title{Subgroup Majorization}
\author{Andrew R Francis}
\address{Centre for Research in Mathematics, University of Western Sydney, Australia}
\author{Henry P Wynn}
\address{Centre for Analysis of Time Series, London School of Economics, UK}

\begin{document}

\begin{abstract}
The extension of majorization (also called the rearrangement ordering), to more general
groups than the symmetric (permutation) group, is referred to as $G$-majorization. There are
strong results in the case that $G$ is a reflection group and this paper builds on this theory in the direction of
subgroups, normal subgroups, quotient groups and extensions. The implications for fundamental cones and order-preserving functions are studied.
The main example considered is the hyperoctahedral group, which, acting on a vector in $\R^n$, permutes and changes the signs of components.


\end{abstract}

\maketitle

\section{Introduction}\label{sec:intro}
Majorization is now the general term for the study of inequalities which began with the theory
of rearrangements expounded at length by Hardy, Littlewood and Polya \cite{hardy1988inequalities} and given impetus by the
book of Marshall and Olkin~\cite{marshall1979inequalities}, now in its second expanded edition~\cite{marshall2010inequalities}. The group of permutations, the symmetric group
$S_n$,  is at the heart of this classical majorization, and a major advance was the extension
to generalised or $G$-majorization which applies particularly to general reflection groups (Eaton and Perlman~\cite{eaton1977reflection}). The present paper
is a contribution to $G$-majorization. Following  a short introduction, we investigate the implication of a number of group operations,
in particular the restriction to subgroups, quotients and extensions.

We begin with the basic definition.
\begin{defn}\label{conv}
Let $\X$ be an $n$-dimensional Euclidean space and let $G$ be a finite matrix  group operating on $\X$.
We define a partial ordering on $\X$, written $y \prec_G x$ by
$$y \in \mathrm{conv} (\O_G(x)).$$
\end{defn}
Here $\mbox{conv}$ is the convex hull and $\O_G(x)=\{gx:  g \in G\}$ is the orbit of $x$ in $\X$ under the action of $G$.
For classical majorization $G$ is the symmetric group $S_n$, and the action of $G$ permutes coordinates. That is, the action of $g\in G$ permutes the entries of $x$.

The following is a basic duality result for $G$-majorization \cite{eaton1984group,giovagnoli1985group}. We use $\langle \cdot, \cdot \rangle$ to denote the Euclidean
inner product and define
$$m(z,x) = \sup_{g \in G} \langle z, g(x)\rangle.$$
\begin{thm} Let $G$ be a closed subgroup of the orthogonal group $O_n$, acting on $\X$. Then
$y \prec_G x$ is equivalent to
$$m(z,y) \leq m(z,x),\;\; \mbox{for all}\;\; z \in \X.$$
\label{supcondition}
\end{thm}
It will be convenient to slightly extend the convex hull definition in Definition~\ref{conv} as part of the discussion
on $G$-majorization as a cone ordering, below.

\subsection{Reflection groups}\label{sec:reflect}
The main results concerning extension of majorization are for the extension from the symmetric group, the classical majorization case, to reflection groups. The essence is contained in Theorem~\ref{bigthm}, below, the major credit for which should go to Eaton and Perlman  and Eaton~\cite{eaton1977reflection,eaton1982review,eaton1984group,eaton1987lectures,eaton1987group,eaton1988concentration,eaton1991concentration}. Giovagnoli and Wynn~\cite{giovagnoli1985group} made contributions working in the context of the extension of majorization to spaces of matrices. These studies realised the importance of the fundamental cone of the reflection groups.  An important question had remained as to whether the equivalent conditions of Theorem~\ref{bigthm} applied {\em only} for reflection groups and this was answered
in the affirmative by Steerneman~\cite{steerneman1990gmajorization} who revisited the theory with careful discussion of many equivalent conditions. Thus the machinery of $G$-majorization was established.

Finite reflection groups acting on Euclidean space are classified according to the finite Coxeter groups, defined by having a generating set $S$ with relations $s^2 = e$, the identity, for all $s \in S$, and $(s_is_j)^{m_{ij}}=e$ for $s_i,s_j\in S$ and with $m_{ij}$ integers $\ge 2$ (see~\cite{Hum90} for instance, for more details).
Any finite reflection group $G$ also has a representation as a subgroup of the orthogonal group $O_n$ acting on $\X=\R^n$, for $n$ sufficiently large. We shall fix $n$ and consider the class $\mathcal G$ of all reflection groups acting in this way on $\X$.

Any $G \in \mathcal G$ is defined by a finite set of distinct generating hyperplanes:
$$V_j = \{x: \langle x, a_j \rangle = 0\},$$
{for $j=1, \ldots, k$,} where the $a_j$ are the positive roots in the root system of $G$.  (Note that we will discuss root systems in a little more detail in Subsection~\ref{subsec: root}.)
These hyperplanes define half spaces
$$V_j^+ = \{x: \langle x, a_j \rangle \geq 0, \; j=1, \ldots, k\}$$
which in turn define the fundamental cone
$$\C_G = \bigcap_{i=1}^k V_j^+.$$
A fundamental {\em region} $R$ has the defining properties (i) $R$ is open,
(ii) for any $x \in R$ there is no other $x' = g(x)\in R$ for any $g \in G$ (equivalently, for $x\in R$ we have $R\cap\O_G(x)=\{x\}$), and
(iii) $\X = \bigcup_{g \in G}\;g (\bar{R}))$, where the bar means closure.
For a finite reflection group $G \in \mathcal G$ the interior of its fundamental cone $\C_G^o$
is a fundamental region.

The fundamental cone is \emph{essential} when $\bigcap_{i=1}^k V_i = {0}$, the origin.
In this case it can be shown that the fundamental region is simplicial, so that there
are exactly $k=n$ defining hyperplanes (see \cite{abramenko2008buildings} Proposition 1.36). The following portmanteau theorem, which applies to the case of an essential cone adapted from Steerneman~\cite{steerneman1990gmajorization}, is given without proof. Following the discussion in that paper the terms ``closed" in the statement of the theorem can be taken as ``essential".
\begin{thm}\label{bigthm}
Let $G$ be a subgroup of $O_n$.
The following are equivalent
\begin{enumerate}[(i)]
\item There is a convex cone $\C$ such that $m(x,y) = \langle x,y \rangle$ for all {$x,y \in \C$}.
\item There is a connected fundamental region unique up to translation under $G$.
\item $G \in \mathcal G$ is a finite reflection group with fundamental cone $\C_G$ and its interior
$\C_G^o$ is a fundamental region.
\item There is a closed convex cone $\C$ such that $y \prec_G x\iff m(z,y) \leq m(z,x)$ for all $z \in \C$.
\item There is a closed convex cone $\C$ such that $y \prec_G x$ is a cone ordering: $x,y \in \C \Rightarrow x-y \in \C^*$, the dual cone of $\C$.
\end{enumerate}
\end{thm}
We shall find part (v) of considerable use. Without loss of generality we state an equivalent version to part (v), namely that
it should hold for representatives $\widetilde{x} = g_1(x),\widetilde{y} = g_2(y) \in \C_G$, for some $g_1,g_2 \in G$ and $\C_G$ the fundamental cone.

In what follows it will not be enough to use only essential cones because there will  be cases where the cone ordering condition is relevant but the cone is not essential. Let us consider a simple case. Suppose that $n=2$ and we are considering the simple group $\{e,g_1\}$ where $e$ is the identity and $g_1: (x_1,x_2) \mapsto (-x_1,x_2)$. The fundamental cone $\C$ is $\{x: x_1 \geq 0\}$, which is inessential. 
From the original Definition~\ref{conv}, $G$-majorization is equivalent to
$$|y_1| \leq |x_1|.$$
But the dual cone is the half-line $\{x_1 \geq 0, x_2=0\}$, and Theorem~\ref{bigthm}(v) breaks down.  We could overcome this difficulty by abandoning the convex hull definition of majorization in Definition~\ref{conv} and adopting the cone condition, without the necessity of the cone being closed (essential). We shall avoid this but it is useful to extend the definition of $G$-majorization and describe the essential and inessential parts of a fundamental cone.

\begin{defn}
Let $G$ be a finite reflection group generated by hyperplanes $\{V_j\}$ with fundamental cone $\C_G = \bigcap_{i=1}^k V_j^+$. Then the inessential part of $\C_G$ is $\C_{G,0} = \bigcap_{i=1}^k V_j$ and the essential is the orthogonal complement
$\C_{G,1} = \C_{G,0}^{\perp} \cap \C_G $.
\end{defn}

For example, in the group $G=\Z_2\times\Z_2$ acting on $\R^2$, generated by the $g_1$ defined above together with 
$g_2: (x_1,x_2) \rightarrow (x_1,-x_2)$, the situation reduces to standard majorization. There are two hyperplanes $V_1$ (the $x_2$ axis) and $V_2$ (the $x_1$ axis), and  $V_i^+$ is the positive half plane $x_i\ge 0$.  The fundamental cone $\C_G$ is the positive quadrant given by $x,_1,x_2\ge 0$; the inessential part is the origin $\C_{G,0}=\{(0,0)\}$; and the essential part is the $\C_{G,1} = \C_{G,0}^{\perp} \cap \C_G = \C_G\setminus\{(0,0)\}$.  In the less trivial case, where $G=\{e,g_1\}$, still acting on $\R^2$, we have just one hyperplane, $V_1$, and the fundamental cone $\C_G$ is given by $x_1\ge 0$ as described above.  The inessential part is the intersection of the hyperplanes (there is only one), namely $V_1$, and the essential part is $\C_{G,0}^{\perp} \cap \C_G = V_1^\perp\cap V_1^+=\{(x_1,0): x_1\ge 0\}$ (note, $V_1^\perp$ is the $x_1$ axis).

If we restrict $G$ and vectors $x,y$ to $\C_{G,0}^{\perp}$, then all the conditions of Theorem~\ref{bigthm} apply. The following extension
of $G$-majorization is based on this.
\begin{defn}
Let $G$ be a finite reflection group. We define essential $G$-majorization by $x \prec^+_G y$ if and only if
$y^+ \in \mbox{conv} (\O(x^+))$, where $x^+,y^+$ are the respective projections of $x,y$ into $\C_{G,0}^{\perp}$ and  $\O(x^+)=\{g(x^+):  g \in G\}$.
\end{defn}
Note that is is not necessary to redefine $G$, because $\mbox{conv} (\O(x^+)) \subseteq \C_{G,0}^{\perp}$, in any case.

It is possible to state the more general version of Theorem~\ref{bigthm}, dropping the requirement that the fundamental cone be closed
and replacing $y \prec_G x $ with $y \prec^+_G x$. In what follows we make the somewhat cavalier assertion that when we use $y \prec_G x $ we have the usual definition of majorization in the essential or $y \prec^+_G x$ in the inessential case.

We are now in a position to recapture matrix descriptions of $G$-majorization stated
simply in terms of inequalities. For this we shall use the cone ordering version Theorem~\ref{bigthm} (v), using a particular choice of the fundamental cone $\C_G$.
Let $\{a_j\} $ be the vectors orthogonally defining the hyperplanes $\{V_j\}$ and let $A = \{a_{ij}\}$ be the matrix whose rows are the  the $a_j$ for  $j=1,\ldots,k$ and (the closure of) the fundamental cone is given by the solution of $$Ax \geq 0.$$
For ease of explanation let us take the case when $k=n$
and $A$ is nonsingular. Then writing $Ax = \delta \geq 0$ we see that
$$x = A^{-1} \delta,$$
and the generators of $\C$ are the columns of $A^{-1}$. The generators of the dual $\C^*$
are the columns of $A$ so that the cone ordering statement $x-y \in \C^*$, for all $x,y \in \C$ becomes
$$x-y = A^T \epsilon,$$
for some $\epsilon \geq 0$. This, in turn is equivalent to
$$
(A^{-1})^T y \leq (A^{-1})^Tx,
$$
or
$$
c_i^T y \leq c_i^Tx
$$
for $x,y \in \C$, with the generators $c_i,\;i=1,\ldots,n$ of $\C^*$.

To summarise, for a general pair $x,y$  it is enough to give the cone ordering  representatives  $\x= g_1(x),\y=g_2(y) \in \C$, for some $g_1,g_2 \in G$, and we have $ y \prec_G x  \iff \x-\y \in \mathcal C^*$, and it is enough to use the generators of $\C^*$ to express this.

When $G$ is the symmetric group $S_n$ operating on $\R^n$, the fundamental cone can be taken as the region given by
$$
x_1 \geq x_2 \geq \cdots \geq x_n,
$$
which is not essential. We map any vector $x= (x_1, \ldots, x_n)^T$ to the reordered values (order statistics)
$\x = (x_{[1]}, \ldots, x_{[n]})^T $ with $x_{[1]} \geq  \cdots \geq x_{[n]}$.

Then
$$
A= \left(
       \begin{array}{rrrrr}
         1 & -1 & \ldots & 0 & 0  \\
         0 & 1 &  -1 & \dots & 0\\
 \vdots &  & & & \vdots\\
       0  &  & \ldots & 1 & -1
       \end{array}
     \right).
$$
Although $A$ is $(n-1) \times n$ we can find the generators of $\C$ by orthogonally completing $A$ to
$$
A_1= \left(
       \begin{array}{rrrrr}
         1 & -1 & \ldots & 0 & 0  \\
         0 & 1 &  -1 & \dots & 0\\
 \vdots &  & & & \vdots\\
       0  &  & \ldots & 1 & -1 \\
       1 & 1 & \ldots &1 & 1
\end{array}
     \right).
$$
Then
$$(A_1^T)^{-1}= \frac{1}{n}\left(
       \begin{array}{ccccc}
         n-1 & -1 & \ldots & -1 & -1  \\
         n-2 & n-2 &  -2 & \dots & -2\\
 \vdots &  & & & \vdots\\
       1  & 1 & \ldots & 1 & -(n-1) \\
       1 & 1 & \ldots &1 & 1
\end{array}
     \right).
$$
Inspecting the rows of $(A_1^T)^{-1}$ and setting $\sum y_i = \sum x_i$, we obtain classical majorization;  otherwise the last inequality is  $\sum y_i \leq \sum x_i$, which gives lower weak majorization.

\section{Example: the hyperoctahedral group}\label{sec:eg.Bn}
The Coxeter group of type $\B_n$, also known as the hyperoctahedral group, is the group of signed permutations of $n$ letters.  It can be represented by $n\times n$ signed permutation matrices, and  is isomorphic to the semidirect product $\Z_2^n\rtimes S_n$, where $S_n$ is the symmetric group on $n$ entries and $\mathbb Z_2^n$ can be interpreted as changing the sign of entries.  The group presentation can be represented by the  Dynkin diagram in Figure~\ref{fig:Bn.dynkin}.
The Dynkin diagram shows the generators $\{s_1,\dots,s_n\}$, and relations $(s_is_j)^{m_{ij}}=e$, where $m_{ij}=3$ if there is a single edge between $s_i$ and $s_j$ and $m_{ij}=4$ if there is a double edge.  The representation of this group as signed permutations has $s_i$ given by the 2-cycle $(i\ i+1)$ for $i=1,\dots,n-1$ and $s_n$ changing the sign of the $n$'th coordinate.  The last generator $s_n$ is often denoted $t$ in the literature on Coxeter groups (sometimes being the sign change on the first coordinate).  For more such information about finite reflection groups, see, for example, Humphreys~\cite{Hum90} or Kane~\cite{kane2001reflection}.
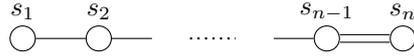
\begin{figure}[ht!]
   \begin{center}\footnotesize
   \begin{tikzpicture}
   \draw (0,0)node[draw,circle](s1){};   \draw(0,.35)node{$s_1$}; %
   \draw (1,0)node[draw,circle](s2){};   \draw(1,.35)node{$s_2$};
   \draw (4,0)node[draw,circle](sn-2){};   \draw(4,.35)node{$s_{n-1}$};
   \draw (5,0)node[draw,circle](sn-1){}; \draw(5,.35)node{$s_{n}$};
   \draw (s1)--(s2);
   \draw (s2)--(1.7,0);
	\draw (4,.05)node[circle](a1){};
	\draw (5,.05)node[circle](b1){};
	\draw (4,-.05)node[circle](a2){};
	\draw (5,-.05)node[circle](b2){};
	\draw (a1)--(b1);
	\draw (a2)--(b2);
   \draw (3.3,0)--(sn-2);
   \draw[thick,dotted] (2.2,0)--(2.8,0);
   \end{tikzpicture}
   \end{center}
\caption{Dynkin diagram for the Coxeter group of type $\B_n$.}
\label{fig:Bn.dynkin}
\end{figure}

We now work through the case  $n=3$.  The extension of the associated orders to $\B_n$ is routine and given in Section~\ref{Bn} below.
The Coxeter group $G$ of type $\B_3$ has Dynkin diagram as shown in Figure~\ref{fig:B3.dynkin}.

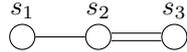
\begin{figure}[ht!]
   \begin{center}\footnotesize
   \begin{tikzpicture}
   \draw (0,0)node[draw,circle](s1){};   \draw(0,.35)node{$s_1$}; %
   \draw (1,0)node[draw,circle](s2){};   \draw(1,.35)node{$s_2$};
   \draw (2,0)node[draw,circle](s3){};   \draw(2,.35)node{$s_3$};
   \draw (s1)--(s2);
	\draw (1,.05)node[circle](a1){};
	\draw (2,.05)node[circle](b1){};
	\draw (1,-.05)node[circle](a2){};
	\draw (2,-.05)node[circle](b2){};
	\draw (a1)--(b1);
	\draw (a2)--(b2);
   \end{tikzpicture}
   \end{center}
\caption{Dynkin diagram for the Coxeter group of type $\B_3$.}
\label{fig:B3.dynkin}
\end{figure}

Its generators $\{s_1,s_2,s_3\}$ can be represented respectively by the following signed permutation matrices operating on $\mathbb R^3$:
$$M_1 = \left(
    \begin{array}{ccc}
      0 & 1 & 0 \\
      1 & 0 & 0 \\
      0 & 0 & 1 \\
    \end{array}
  \right),\;\;
  M_2 = \left(
    \begin{array}{ccc}
      1 & 0 & 0 \\
      0 & 0 & 1 \\
      0 & 1 & 0 \\
    \end{array}
  \right),\;\;
  M_3 = \left(
    \begin{array}{rrr}
      1 & 0 & 0 \\
      0 & 1 & 0 \\
      0 & 0 & -1 \\
    \end{array}
  \right).
$$

The fundamental cone consistent with the ordering in the Dynkin diagram is
$$\C_G = \{x= (x_1,x_2,x_3)^T: x_1 \geq x_2 \geq x_3 \geq 0\}.$$
The $2$-dimensional supporting hyperplanes of $\C_G$ are given by the equations
$$x_1-x_2=0,\;\;x_2-x_3 =0\;\;\text{and }\; x_3=0.$$
The fundamental cone is the region defined by the inequalities
$$x_1-x_2\geq 0,\;\;x_2-x_3  \geq 0\;\;\text{ and }\; x_3 \geq 0,$$
and the matrix $A$ and its inverse are given by
$$ A = \left(
    \begin{array}{rrr}
      1 & -1 & 0 \\
      0 & 1 & -1 \\
      0 & 0 & 1 \\
    \end{array}
  \right)
\qquad\text{and}\qquad
A^{-1} = \left(
    \begin{array}{rrr}
      1 & 0 & 0 \\
      1& 1 &  0 \\
      1 & 1 & 1 \\
    \end{array}
  \right).
$$
The representative $\x$ of $x\in\R^3$ in this cone is then found by arranging the coordinates in weakly decreasing order  according to their  absolute values.  We will denote the re-ordered coordinates
\[\x=(x_{[1]},x_{[2]},x_{[3]}),\]
so that $x_{[1]}$ is the coordinate with the largest absolute value, $x_{[2]}$ is the coordinate with the next largest absolute value and so on.  In other words,
$|x_{[1]}| \geq |x_{[2]}| \geq |x_{[3]}| \geq 0$.

We now have an induced order $y\prec x$ given by
\begin{align*}
|y_{[1]}| &\leq |x_{[1]}|\\
|y_{[1]}|+|y_{[2]}|&\le |x_{[1]}|+|x_{[2]}|\\
|y_{[1]}|+|y_{[2]}| +|y_{[3]}|&\leq |x_{[1]}|+|x_{[2]}|+ |x_{[3]}|,
\end{align*}
obtained from the columns of $A^{-1}$ as described in Section~\ref{sec:intro}. This is lower weak majorization
on the absolute values. 

\section{Subgroup and group extension constructions}\label{sec:quotients}

If $N$ is a normal subgroup of $G$ and $H$ is a subgroup of $G$ isomorphic to $G/N$ we say that $G$ is an extension of $N$ by $H$.  For general $N\lhd G$ it is not always the case that the quotient $G/N$ is isomorphic to a subgroup of $G$ (for example the quaternion group, its normal subgroup $\{\pm1\}$ and quotient $\Z_2\times\Z_2$), so the case of a group extension provides a special infrastructure for majorization.

We begin by describing how $G$-majorization can be restricted to a majorization by a subgroup $H$ of $G$.

Let $\{\X,G\}$ define a $G$-majorization and let $H$ be a subgroup of $G$ (not necessarily normal).  We define $ y \prec_H x$ formally as
$$y \in \mbox{conv}( \O_H(x)).$$
We have
\begin{equation}\label{eq:refinement}
y \prec_H x \implies y \prec_G x,
\end{equation}
because $H \leq G \implies \mbox{conv}( \O_H(x)) \subseteq \mbox{conv}( \O_G(x))$. We can say that $\prec_G$ is a {\em refinement}
of $\prec_H$.
We can give an instructive proof of Eq.~\eqref{eq:refinement} using the equivalent condition from Theorem~\ref{supcondition}.
Thus,
$$\sup_{g \in G} \langle z, g (y)\rangle = \sup_{g \in G} \langle z, gg'(y) \rangle $$
for any  fixed $g'\in G$. And similarly for $x$,
$$\sup_{g \in G} \langle z, g (x)\rangle = \sup_{g \in G} \langle z, gg''(x) \rangle $$
for any fixed $g'' \in G$.
Now let the right cosets of $H$ be $Hg_1, H g_2, \ldots$ . Then
 $$\sup_{g \in G} \langle z, gg'(y) \rangle = \sup_i \sup_{h \in H} \langle z, h g_i g'(y) \rangle,$$
 and suppose the $\sup_i$ is achieved at $i=r$. Then take $g' = g_r^{-1}$, and the last expression
 reduces to $\sup_{h \in H} \langle z, h (y) \rangle$. Carrying out  a similar procedure with $x$ and appealing
 to $y \prec_H x$ gives the result.

If $G$ is an extension of $N$ by $H$ then we can apply this same construction to produce a majorization by the quotient $G/N$.  In this case $H$ is isomorphic to $G/N$, but the majorization depends on the isomorphism.  A convenient way to approach this is to extend this isomorphism $H\cong G/N$ to a homomorphism $G\to G$ with kernel $N$.  This can always be done, as the following (textbook) Lemma shows:
\begin{lem}
Let $G$ be a group, $N\lhd G$ and $H\le G$.  If $\pi:G/N\to H$ is an isomorphism then $\pi$ extends to a homomorphism $\phi:G\to G$ with kernel $\ker\phi=N$.  Furthermore, $\mathrm{im}\,\phi=H\cong G/N$.
\end{lem}
\begin{proof}
For $g\in G$ define $\phi(g):=\pi(gN)$.  If $n\in N$ then $\phi(n)=\pi(N)=1$, since $\pi$ is a homomorphism ($N$ is the identity of $G/N$), and so $N\subseteq\ker\phi$.  Conversely if $g\in\ker\phi$ then $\phi(g)=\pi(gN)=1$, but $\pi$ is an isomorphism so this implies $gN=N$ and therefore $g\in N$, completing the proof of the main statement.  The claim that im$\,\phi\cong G/N$ is immediate from the first isomorphism theorem.
\end{proof}

This shows that if $G$ is an extension of $N\lhd G$ by $H\le G$ then there is a homomorphism $\phi:G\to G$ such that $\ker\phi=N$ and im$\,\phi=H$.  Different choices of the homomorphism $\phi$ may provide different subgroups $H=\,$im$\,\phi$, each isomorphic to $G/N$.  To define a majorization with respect to $G/N$, we therefore need to take into account the map $\phi$.
In the same way that we like to consider $G$ as a matrix group acting on $\X$, we can use a matrix representation $(G/N,\phi)$ of $G/N$ which depends on $\phi$. In this way there is a natural definition of majorization for $G/N$, depending on $(N,\phi)$:
$$ y \prec_{(G/N,\phi)} x\quad \Longrightarrow\quad   y \in \mbox{conv}\left(\O_{(G/N,\phi)}\right).$$

Just as it is natural to consider $G$ acting on $\X$ as a matrix group, we can consider $G$ as the product group
$$G/N \times N \cong G.$$
We repeat, whereas the representation for $N$ is simply induced by $G$,  that
for $G/N$, and consequently the majorization, depends on the particular $\phi$ chosen.

Let $\V=\bigcup V_i$ be the union of the set of reflecting hyperplanes $V_i$ defined by the finite reflection group $G$ acting on the space $\X$.
Let  $\C^\circ_G$ denote the fundamental region corresponding to $G$ (the interior of the fundamental cone $\C_G$), as defined above.  This is an open convex set with the property that each orbit of $x\in\X\setminus\V$ contains exactly one element $gx$ (for some $g\in G$) in $\C^\circ_G$.
The set of translates $\{g\C^\circ_G\mid g\in G\}$ of the fundamental region is pairwise disjoint, and its union is $\X\setminus\V$.

In the case that $G$, $N$ and $G/N$ are reflection groups, we have a very simple relationship between their fundamental cones:

\begin{thm}\label{fund} Suppose $G$ is an extension of $N$ by $H$, and that $G$, $N$ and $H\cong G/N$ are all reflection groups.
Then
\[\C^\circ_G=\C^\circ_N\cap\C^\circ_{H}.\]
\end{thm}

\begin{proof}
First note that $\C^\circ_G$ is entirely contained within $\C^\circ_N$, since actions under elements of $N$ are also actions of elements of $G$.

Consider the images of the fundamental region $\C^\circ_G$ under the action of elements of $N$.  Since $N\le G$ this action translates $\C^\circ_G$ into the $|N|$ disjoint translates of $\C^\circ_N$.  That is, for $n\in N$, $n\C^\circ_G\subseteq n\C^\circ_N$.

Now consider $g\in G\setminus N$, chosen so that $g\in gN\neq N$.  The action of $g$ on $\C^\circ_G$ must translate it to one of the $|N|$ regions $\{n\C^\circ_G\mid n\in N\}$, since the union of these regions is the whole of $\X\setminus\V$.  Then $\{g\C^\circ_G\mid g\in gN\}$ is a set of translates of $\C^\circ_G$, exactly one of which is in each of $\{n\C^\circ_N\mid n\in N\}$. For, suppose $gn_1\C^\circ_G$ and $gn_2\C^\circ_G$ are in the same $n\C^\circ_N$. Then $n_1\C^\circ_G$ and $n_2\C^\circ_G$ are in the same $n\C^\circ_N$ and therefore $n_1=n_2$.
That is, for each coset $gN$ and each $N$-translate $n\C_N$ there is a unique representative $g'\in gN$ with the property that $g'\C^\circ_G\subseteq n\C^\circ_N$.

Consider now the fundamental cone $\C_{H}$.  We claim that this is equal to the union of $N$-translates of $\C_G$, that is
\[
\C_{H}=\bigcup_{n\in N}n\C_G.
\]
This follows because every element of $G$ can be written uniquely as a product of an element of $H$ with an element of $N$, so that
\[\bigcup_{h\in H}h\bigcup_{n\in N}n\C_G=\bigcup_{g\in G}g\C_G=\X,\]
and because
\[h\C_G^\circ\cap\C_G^\circ=\varnothing\]
for any non-identity $h\in H$.

As a consequence, we have that
\[\C^\circ_N\cap\C^\circ_{H}=\C^\circ_N\cap\left(\bigcup_{n\in N}n\C_G\right)^\circ.\]
But as noted above, there is a unique $N$-translate of $\C_G^\circ$ inside $\C_N^\circ$, namely $\C_G^\circ$ itself, and for all other $n\neq e$ in $N$ we have $n\C_G\cap\C_N^\circ=\varnothing$.  Therefore
\[
\C^\circ_N\cap\left(\bigcup_{n\in N}n\C_G\right)^\circ
=\C_N^\circ\cap\C_G^\circ=\C_G^\circ\]
since $\C_G^\circ\subseteq\C_N^\circ$, as required.
\end{proof}

In \cite{niezgoda2001norm,niezgoda2002structure} the authors discuss a method of constructing
larger Eaton triples by considering union of cones associated with smaller Eaton triples. An
Eaton triple is an object which satisfies slightly weaker conditions than
in Theorem~\ref{bigthm}. They take the intersection of the cones
from the Eaton triples and the groups generated by the union of the groups from the Eaton triples.
Although our theorem above is restricted to reflection groups it is
otherwise quite general and reveals the importance of the normal subgroup
property. This property facilitates the study of general classes of
refections groups and subgroups.

The extensive study by  Maxwell~\cite{maxwell1998normal} shows that all normal subgroups of a finite reflection group are either of index 2 in the group, or are also finite reflection groups, so that the conditions of the Theorem \ref{fund} are very often satisfied.  Notable exceptions include the alternating subgroup $A_n$ as a normal subgroup of the symmetric group $S_n$:  the alternating group is not a reflection group (but it is of index 2 in $S_n$).

\section{Normal subgroups in the hyperoctahedral group}

The normal subgroups of the group $G$ of type $\B_n$ (and other finite and affine reflection groups) are described in Maxwell~\cite{maxwell1998normal}.
For instance, the subgroup of type $\A_n$ (the symmetric group $S_{n+1}$) and the subgroup $\Z_2^n$ are both normal in $G$, and have quotients $G/N\cong\Z_2$ and $S_3$ respectively.

In the case $n=3$, one composition series of $G$ is as follows:
\begin{center}
\begin{tikzpicture}
\draw (0,0) node (G) {$G$};
\draw (2,0) node (S4) {$S_4$};
\draw (4,0) node (A4) {$A_4$};
\draw (6,0) node (Z2Z2) {$\Z_2^2$};
\draw (8,0) node (Z2) {$\Z_2$};
\draw (10,0) node (1) {$1$.};
\draw[->,>=latex] (G)--(S4);
\draw[->,>=latex] (S4)--(A4);
\draw[->,>=latex] (A4)--(Z2Z2);
\draw[->,>=latex] (Z2Z2)--(Z2);
\draw[->,>=latex] (Z2)--(1);
\footnotesize
\draw (1,.3)node {$\Z_2$};
\draw (3,.3)node {$\Z_2$};
\draw (5,.3)node {$\Z_3$};
\draw (7,.3)node {$\Z_2$};
\draw (9,.3)node {$\Z_2$};
\end{tikzpicture}
\end{center}
Here $S_4$ is the symmetric group on 4 letters and $A_4$ is the alternating group on 4 letters (the group of even permutations).  The labels on the arows indicate the composition factors, so that for instance $S_4\lhd G$ and $G/S_4\cong\Z_2$.  The composition factors of a group are unique up to isomorphism and order in the series, by the Jordan-H\"older Theorem.   However there are normal subgroups that do not have simple factors and so are not featured in the composition series.  For instance, $\Z_2^3\lhd G$ and $G/\Z_2^3\cong S_3$.

In this section we develop a detailed example for the case $n=3$ in relation to these two normal subgroups ($S_4$ and $\Z_2^3$), including deriving the partial orders resulting from the $G$-majorization described above.

\subsection{The normal subgroup of type $\A_3$}\label{subsec:A3}

The normal subgroup $N$ of type $\A_3$ (the symmetric group $S_4$) is generated by the elements $\{s_1,s_2,s_3s_2s_3\}$
and has Dynkin diagram as shown in Figure~\ref{fig:A3.dynkin}.

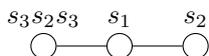
\begin{figure}[ht!]
   \begin{center}\footnotesize
   \begin{tikzpicture}%
   \draw (0,0)node[draw,circle](s1){};   \draw(0,.35)node{$s_3s_2s_3$}; %
   \draw (1,0)node[draw,circle](s2){};   \draw(1,.35)node{$s_1$};
   \draw (2,0)node[draw,circle](s3){};   \draw(2,.35)node{$s_2$};
   \draw (s1)--(s2) (s2)--(s3);
   \end{tikzpicture}
   \end{center}
\caption{Dynkin diagram for the normal subgroup of type $\A_3$ in the Coxeter group of type $\B_3$.}
\label{fig:A3.dynkin}
\end{figure}

The representations are given by the matrices $M_1$ and $M_2$ as for $G$ (given in Section~\ref{sec:eg.Bn}) but with $M_3$ replaced by
$$
  M'_3 = M_3M_2M_3 =  \left(
    \begin{array}{rrr}
      1 & 0 & 0 \\
      0 & 0 & -1 \\
      0 & -1 &  0 \\
    \end{array}
  \right).
$$
The relations between $M_1$, $M_2$ and $M_3'$ indicated by the Dynkin diagram are easily checked.
The $2$-dimensional supporting hyperplanes of $\C_N$ are given by
$$x_1-x_2=0,\;\;x_2-x_3 =0,\;\;\; x_2+x_3=0,$$
the fundamental cone is
$$x_1-x_2 \geq 0,\;\;x_2-x_3  \geq 0,\;\;\; x_2+x_3 \geq 0,$$
$$
  A =  \left(
    \begin{array}{rrr}
      1 & -1 & 0 \\
      0 & 1& -1 \\
      0 & 1 &  1 \\
    \end{array}
  \right).
$$
The representative of $x\in\X$ in the cone is $\x=(|x_{[1]}|,|x_{[2]}|,x_{[3]})$ with
\[|x_{[1]}|\ge|x_{[2]}|\ge x_{[3]}.\]
Now
$$
  (A^T)^{-1} =  \left(
    \begin{array}{rrr}
      1 & 0 & 0 \\
      \frac{1}{2} &  \frac{1}{2}&  -\frac{1}{2} \\
       \frac{1}{2} &  \frac{1}{2} &   \frac{1}{2} \\
    \end{array}
  \right).
$$
and from the rows of $(A^T)^{-1}$ we have that $y \prec_Nx$ becomes
\begin{align*}
|y_{[1]}| & \leq  |x_{[1]}| \\
|y_{[1]}| + |y_{[2]}| -y_{[3]}     & \leq |x_{[1]}| + |x_{[2]}|  -x_{[3]}\\
|y_{[1]}| + |y_{[2]}|  +y_{[3]} & \leq |x_{[1]}| + |x_{[2]}|  +x_{[3]}
\end{align*}

The subgroup $G/N$ is isomorphic to $\Z_2$, and following the discussion in Section~\ref{sec:quotients} we are free, up to isomorphism, to select $\phi$ consistent with this quotient operation.
There are various options. We can make it {\em dependent} on the selection of generators
for $G$ or $N$. For example, we could take the reflection in $x_3 = 0$ as the non-identity group element of $G/N$.  For this the additional order is
 $$|y_3| \leq |x_3|.$$
But this choice seems somewhat arbitrary, we could have used $x_1$ or $x_2$, but in any such cases there would also be a preferred ``direction". We
prefer the interesting case where the reflection generating $\Z_2$ is through $\{x: \sum_i x_i =0\}$ which would lead to
$$ \left|\sum_{i=1}^3 y_i\right| \leq \left|\sum_{i=1}^3 x_i\right|.$$

\subsection{The normal subgroup $\Z_2^3$}\label{subsec:Z2}

The normal subgroup $N\cong\Z_2^3$ can be generated by the commuting reflections $\{s_1s_2s_3s_2s_1, s_2s_3s_2, s_3\}$.  These are sign changes in the first, second and third coordinates respectively.  Because they commute with each other, they correspond to the rather uninteresting disconnected Dynkin diagram shown in Figure~\ref{fig:Z2^3.dynkin}.
\begin{figure}[ht!]
   \begin{center}\footnotesize
   \begin{tikzpicture}%
   \draw (0,0)node[draw,circle](s1){};   \draw(0,.35)node{$s_1s_2s_3s_2s_1$}; %
   \draw (2,0)node[draw,circle](s2){};   \draw(2,.35)node{$s_2s_3s_2$};
   \draw (4,0)node[draw,circle](s3){};   \draw(4,.35)node{$s_3$};
   \end{tikzpicture}
   \end{center}
\caption{Dynkin diagram for the normal subgroup $\Z_2^3$ of the group of type $\B_3$.}
\label{fig:Z2^3.dynkin}
\end{figure}
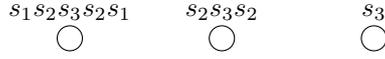

This abelian subgroup is the kernel of the map $\phi:G\to G$ that sends $s_1\mapsto s_1$, $s_2\mapsto s_2$ and $s_3\mapsto 1$.  Then $G/N\cong\text{im}\phi=\langle s_1,s_2\rangle\cong S_3$.  The reflecting hyperplanes for $N$ are simply the 2-dimensional planes orthogonal to the coordinate axes, given by $x_1=0$, $x_2=0$ and $x_3=0$.
The fundamental cone is then the positive octant of $\R^3$ given by $x_1\ge 0,\ x_2\ge 0,\ x_3\ge 0$.  For any point $x=(x_1,x_2,x_3)\in\R^3$ its representative in the cone is simply $\x=(|x_1|,|x_2|,|x_3|)$, and for any other $y\in\R^3$ we have the corresponding order $y\prec_N x$ given by the inequalities $|y_1|\le |x_1|$, $|y_2|\le |x_2|$ and $|y_3|\le |x_3|$.

The reflecting hyperplanes for the quotient $G/N$ are given by $x_1-x_2=0$ and $x_2-x_3=0$, so that this fundamental cone is $x_1\ge x_2\ge x_3$.  Note that this cone is not essential and in particular is not contained in any of the octants of the space defined by the coordinate axes.  Let $x=(x_{[1]},x_{[2]},x_{[3]})$ be given by ordering the coordinates so that $x_{[1]}\ge x_{[2]}\ge x_{[3]}$ and we have the lower weak 
majorization discussed above.

\subsection{Inequalities for the group of type $\B_n$ (and $\D_n$)}\label{Bn}

The inequalities in both the previous subsections extend in a straightforward manner to the general case when $G$ is of type $\B_n$.
The ordering for the group $G$ is given by $y\prec_G x$ if and only if
$$
\sum_{i=1}^j  y_{[i]} \leq \sum_{i=1}^j  x_{[i]}, \;j=1,\ldots,n. $$

The $n=3$ example of a subgroup of type $\A_3$ in the group of type $\B_3$ does not generalize to a subgroup of type $\A_n$ but rather to one of type $\D_n$.  This group has Dynkin diagram as shown in Figure~\ref{fig:Dn.dynkin}.

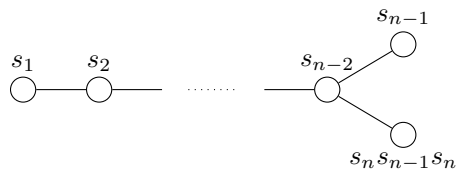
\begin{figure}[ht!]
   \begin{center}\footnotesize
   \begin{tikzpicture}%
   \draw (0,0)node[draw,circle](s1){};   \draw(0,.35)node{$s_1$}; %
   \draw (1,0)node[draw,circle](s2){};   \draw(1,.35)node{$s_2$};
   \draw (2,0)node[circle](s3){};   %
   \draw (3,0)node[circle](sn-3){};   %
   \draw (4,0)node[draw,circle](sn-2){};   \draw(4,.35)node{$s_{n-2}$};
   \draw (5,.6)node[draw,circle](sn-1){};   \draw(5,.95)node{$s_{n-1}$};
   \draw (5,-.6)node[draw,circle](sn){};   \draw(5,-.95)node{$s_ns_{n-1}s_n$};
   \draw (s1)--(s2) (s2)--(s3) (sn-3)--(sn-2)--(sn-1) (sn-2)--(sn);
   \draw[dotted] (s3)--(sn-3);
   \end{tikzpicture}
   \end{center}
\caption{Dynkin diagram of the type $\D_n$ Coxeter group, with generators showing its embedding as a normal subgroup of the group of type $\B_n$.}
\label{fig:Dn.dynkin}
\end{figure}

When $n=3$ this diagram reduces to the three nodes on the right hand side, and hence the isomorphism with the group of type $\A_3$ in that small case (see Figure~\ref{fig:A3.dynkin}).  The ordering derived from the normal subgroup in the general type $\B$ case then gives us a set of ``type $\D$" inequalities following from those  we have already obtained.
For  $N\lhd G$  of type $\D_n$ the quotient is $G/N\cong \Z_2$, and the order $y\prec_N x$ is given by the inequalities:
\begin{align*}
\sum_{i=1}^j |y_{[i]}| &\leq \sum_{i=1}^j |x_{[i]}|,\quad \text{for }j=1,\ldots,n-2,\\
\sum_{i=1}^{n-1}|y_{[i]}|  -y_{[n]}  &\leq \sum_{i=1}^{n-1} |x_{[i]}| -x_{[n]},\\
\sum_{i=1}^{n-1}|y_{[i]}|  +y_{[n]} &\leq \sum_{i=1}^{n-1} |x_{[i]}|+ x_{[n]}.
\end{align*}
For any $n$ the  subgroup $G/N$ is $\Z_2$, and for the appropriate choice of generator gives our preferred version of the inequality:
$$\left| \sum_{i=1}^n y_i\right| \leq \left|\sum_{i=1}^n x_i \right|$$
Finally, when $N=\Z_2^n\lhd G$, the order $y\prec_N x$ is given by the inequalities
\[
|y_1|\le |x_1|,\quad
|y_2|\le |x_2|,\quad
\cdots,\quad
|y_n|\le |x_n|
\]
and $G/N$ is the symmetric group {$S_n$} applied to $X$ giving the inessential (lower weak) version of majorization.

\section{Order-preserving functions}
A major motivation for the study of majorization is to state inequalities for functions
of interest in different fields. Formally, this means considering order-preserving functions.
\begin{defn}
An order preserving function associated with a $G$-majorization is a function $f$ such
that
$$y \prec_G x \; \implies f(y) \leq f(x)$$
\end{defn}
We label the set of all such order preserving functions $\mathcal F_G$. If $H$ is a proper subgroup of $G$ then $\mathcal F_G \subset \mathcal F_H$.  It is also clear, since $x \prec_G gx$
and $gx \prec_G x$ that any $G$-order preserving function $f$ is $G$-invariant: $f(x) = f(gx),\;\;\mbox{for all} \;\; g \in G$.

Now, as above, consider a non-trivial normal subgroup $N$ and the quotient subgroup $H \cong G/N$ (so that $G$ is an extension of $N$ by $H$). For the latter we adopt one
representation given by a choice of the homomorphism $\phi$. We have, applying the subgroup property twice,
\begin{equation}\label{subset}
\mathcal F_G \; \subset \; \mathcal F_N \cap \mathcal F_H.
\end{equation}

While we explain below that the reverse inclusion may not hold, it is nevertheless easily shown that a function that is order-preserving with respect to both $N$ and $H$ is also $G$-invariant.

\begin{lem}\label{lem:f.G.inv}
If $f \in \mathcal F_N \cap \mathcal F_H$ then $f$ is $G$-invariant.
\end{lem}
\begin{proof}
If $g \in G$ then $g$ can be written $g=nh$ for some $n \in N$ and $h \in H$. Then, using the fact that
$\prec_N$ and $\prec_H$ are order preserving functions, respectively, $N$- and $H$-invariant, we have:
\begin{align*}
  f(gx) &  =  f(hnx) \\
   & =  f(nx) \\
  & =   f(x).
\end{align*}
\end{proof}

The subset inclusion in \eqref{subset} may be strict. We see this as follows. As before $ y \prec_G x $ is equivalent to
$ \x,\y \in \C_G \;\; \mbox{and} \;\; \x-\y \in \C_G^*,$
where $\x,\y$ are representatives of $x,y$, respectively, in $\C_G$. But by Theorem \ref{fund}, and assuming
we have closed essential cones we have:
\begin{align*}
\C^*_G & = (\C_N \cap \C_H)^* \\
       & =  \C_N^* + \C_H^* \\
       & =\mbox{conv}(\C_N^* \cup \C_H^*),
\end{align*}
where the ``+" is the Minkowski sum and ``conv" is the convex hull. Equality in (\ref{subset}) holds if and only if
$$y \prec_G x \implies \{y \prec_H x \} \vee \{y \prec_N x \}$$
By the above this holds if and only if
\begin{equation}\label{equalcase}
\mbox{conv}(\C_N^* \cup \C_H^*) = \C_N^* \cup \C_H^*,
\end{equation}
which, in turn,  holds if and only if $\C_N^* \cup \C_H^*$ is convex. Translating this to order-preserving functions,
equality in \eqref{subset} holds if and only if this convexity holds. One way of seeing when condition \eqref{equalcase} breaks down is that the set of inequalities which give $y \prec_G x$
is simply not obtained by listing the inequalities from $\prec_N$ and $\prec_H$. If $\C_N^* \cup \C_H^*$ is strictly contained in
$\mbox{conv}(\C_N^* \cup \C_H^*)$ there are more pairs $x,y$ to compare and fewer functions $f$ satisfying $f(y) \leq f(x)$. This is the case, for example, in Section~\ref{subsec:A3}.

It should be mentioned that \cite{niezgoda2001norm,niezgoda2002structure} note the importance of
strict inclusion in their version of Equation~\eqref{equalcase}, and refer to the
relationship of their subgroups and cones in this case as being {\em effective}. 

\subsection{Root systems}\label{subsec: root}
We need to collect some basic material about root systems to understand futher interplay between the inequalities defining the majorization for $G,N$ and $H=G/N$, the trio of this paper. Let us return to the example of the groups of type $\B_3$
in Section~\ref{sec:eg.Bn}.  We saw that the fundamental cone $\C$ is defined by the hyperplanes $\{x_1-x_2=0 , x_2-x_3=0, x_3=0\}$.  These are sometimes referred to as the walls of $\C$. The dual cone is generated by the vectors orthogonal to these hyperplanes, namely
$$a_1=\{ (1,-1,0), \; a_2=(0,1,-1),\;a_3=(0,0,1).$$
Writing $e_1,e_1,e_3$ for the unit vectors $(1,0,0),(01,0),(0,0,1)$ respectively,  the generators can be written as
$e_1-e_2, e_2-e_3, e_3.$
These are referred to as the {\em fundamental root system} of the group. For type $\B_n$ the system is
$$\{e_1-e_2, e_2-e_3, \ldots, e_{n-1}-e_n,e_n\},$$
Since the generators of the dual cone applied to representative vectors in $\C$ define the majorization,
and we can take these these generators as the fundamental roots, we can study $\prec_G, \prec_N$ and $\prec_H$ via their root systems.

From Theorem \ref{fund} it must be the case that hyperplanes defining  the walls of $\C_G$ are comprised of walls from $\C_N$ and $\C_H$, and hence it must be the case that the fundamental roots of $G$ must comprise certain
roots from $N$ and $H$.  We see this clearly from Subsection~\ref{subsec:A3}. There we see that the roots of the group of type $\A_3$ are
$$\{e_1-e_2,e_2-e_3, e_1+e_2\},$$
and with our selection of $x_3=0$ as the wall of $\A_3$ we see that
$$\{e_1-e_2, e_2-e_3, e_3\} \subset \left( \{e_1-e_2,e_2-e_3, e_1+e_2\} \cup \{e_3\}\right),$$
confirming our proposition. For the example in Subsection~\ref{subsec:Z2}, we have
$$\{e_1-e_2, e_2-e_3, e_3\} \subset \left(\{e_1-e_2,e_2-e_3, e_1+e_2\} \cup \{e_1,e_2,e_3\}\right).$$

These cases provide counterexamples to confirm the strict inclusion in \eqref{subset}.
Thus in the first case above one can easily check that $e_1+2e_3$ lies in $\C_G$ but in neither $\C^*_N$ nor $\C^*_H$.

\subsection{Differential conditions for $G$ order preserving functions}
Root systems are the key to the differential condition for $\prec_G$ preserving functions. Thus, take $x,y$ with representatives $\x, \y \in \C_G$ with $\x= \y + \epsilon$. Then $y \prec_G x$ if and only if $\epsilon = \x - \y \in \C^*$. Let $f$ be a continuously differentiable $\prec_G$ preserving function. Write
$$
f(\x) = f(\y) + \langle \frac{\partial f}{\partial x}, \epsilon\rangle ||\epsilon||+ o(||\epsilon||),
$$
where $\frac{\partial f}{\partial x} = \left( \frac{ \partial f}{\partial x_1}, \ldots, \frac{\partial f}{\partial x_n} \right)^T$, the gradient.
Letting $||\epsilon|| \rightarrow 0$ we see that a necessary and sufficient condition for $f(\y) \leq f(\x)$ is that
$$ \langle \frac{\partial f}{\partial x}, a_i\rangle\geq 0,$$
for all fundamental roots $a_i$.

For the case of type $\B_n$ above, the conditions are (for the representatives):
$$ \frac{\partial f}{\partial x_{i}} - \frac{\partial f}{\partial x_{i+1}} \geq 0, \;(i=1,\ldots n-1),\;\;\frac{\partial f}{\partial x_n} \geq 0$$
on the cone $\left\{|x_1| \geq |x_2| \ldots \geq |x_n| \geq 0 \right\}$.  The invariant polynomial ring (see eg \cite[Section 16]{kane2001reflection}) has basis
\begin{align*}
g_k  & = \sum_{1\leq  i_1 <\dots< i_k \leq n} x^2_{i_1} \cdots x^2_{i_k},\quad\text{for }k=1,\ldots,n,\text{ and} \\
 h &  = x_1 \cdots x_n.
\end{align*}
As an example consider invariants of the form
$$f = a g_1+b h.$$
A little analysis shows that $f$ is $G$ order preserving with  $G$ of type $\D_n$ (for all $x \in \R^n$) if and only if $2a \geq b \geq 0$.  

In the essential case when the $A$-matrix is invertible we have a concise matrix expression for a $G$ invariant $f$ to be $\prec_G$  preserving:
\begin{enumerate}
\item $\prec_G$ is equivalent to $(A^{-1})^T y \leq (A^{-1})^T x \leq 0$ for representatives  $Ax \geq 0,\; Ay \geq 0$.
\item $A \; \frac{\partial f}{\partial x} \geq 0$ for representatives $Ax \geq 0, \; Ay \geq 0$ .
\end{enumerate}

Since, the group of type $\D_n$ is a subgroup of that of type $\B_n$, we have for their order preserving
functions $F_{\B_n} \subset F_{\D_n}$. To confirm the inclusion is strict we give an example
function in $F_{\D_n}$ which is not in $F_{\B_n}$. Consider the function for $n=4$ given by:
$$f= \frac{1}{4} (x_1^2+x_2^2+x_3^2+x_4^2)^2 - |x_1x_2x_3x_4|$$
This function is invariant under both types $\B_n$ and $\D_n$. The first three derivative tests
are the same for $\prec_{B_n}$ and $\prec_{D_n}$. For $\D_n$ we  confirm that
$$\frac{\partial f}{\partial x_3}+\frac{\partial f}{\partial x_4} = (x_3+x_4)(x_1^2+x_2^2+x_3^2+x_4^2- x_1x_2) \geq 0$$
holds on $\C_{\D_n} = \{x:x_1\geq x_2 \geq x_3 \geq x_4 \geq,\; x_3+x_4 \geq 0\}$. For $\B_n$ we should have
$$\frac{\partial f}{\partial x_4} = (x_1^2+x_2^2+x_3^2+x_4^2)x_4-x_1x_2x_3 \geq 0,$$
on $\C_{\B_n} = \{x:x_1\geq x_2 \geq x_3 \geq x_4 \geq 0\}$. But this fails, for example at $x=(1,1,1,\frac{1}{4})$.

\medskip

\section*{Acknowledgements}
ARF was supported by Australian Research Council Future Fellowship FT100100898.
HPW was supported by a Fellowship from the Leverhulme Trust.

\bibliographystyle{plain}

\begin{thebibliography}{10}

\bibitem{abramenko2008buildings}
P.~Abramenko and K.S. Brown.
\newblock {\em Buildings: theory and applications}.
\newblock Springer, 2008.

\bibitem{eaton1982review}
M.L. Eaton.
\newblock A review of selected topics in multivariate probability inequalities.
\newblock {\em The Annals of Statistics}, pages 11--43, 1982.

\bibitem{eaton1984group}
M.L. Eaton.
\newblock On group induced orderings, monotone functions, and convolution
  theorems.
\newblock In {\em Inequalities in statistics and probability: proceedings of
  the Symposium on Inequalities in Statistics and Probability, October 27-30,
  1982, Lincoln, Nebraska}, volume~5, page~13. Michigan State Univ, 1984.

\bibitem{eaton1987group}
M.L. Eaton.
\newblock {\em Group induced orderings with some applications in statistics}.
\newblock University of Minnesota, School of Statistics, 1987.

\bibitem{eaton1987lectures}
M.L. Eaton.
\newblock {\em Lectures on topics in probability inequalities}.
\newblock Number~35. Centrum voor Wiskunde en Informatica, 1987.

\bibitem{eaton1988concentration}
M.L. Eaton.
\newblock Concentration inequalities for {Gauss-Markov} estimators.
\newblock {\em Journal of Multivariate Analysis}, 25(1):119--138, 1988.

\bibitem{eaton1977reflection}
M.L. Eaton and M.D. Perlman.
\newblock Reflection groups, generalized {S}chur functions, and the geometry of
  majorization.
\newblock {\em The Annals of Probability}, 5(6):829--860, 1977.

\bibitem{eaton1991concentration}
M.L. Eaton and M.D. Perlman.
\newblock Concentration inequalities for multivariate distributions: I.
  multivariate normal distributions.
\newblock {\em Statistics \& probability letters}, 12(6):487--504, 1991.

\bibitem{giovagnoli1985group}
A.~Giovagnoli and HP~Wynn.
\newblock {$G$}-majorization with applications to matrix orderings.
\newblock {\em Linear Algebra and its Applications}, 67:111--135, 1985.

\bibitem{hardy1988inequalities}
G.H. Hardy, J.E. Littlewood, and G.~Polya.
\newblock {\em Inequalities}.
\newblock Cambridge University Press, 1988.

\bibitem{Hum90}
James~E. Humphreys.
\newblock {\em Reflection groups and {C}oxeter groups}.
\newblock Cambridge University Press, Cambridge, 1990.

\bibitem{kane2001reflection}
R.~Kane.
\newblock {\em Reflection groups and invariant theory}.
\newblock Springer, 2001.

\bibitem{marshall1979inequalities}
A.W. Marshall and I.~Olkin.
\newblock {\em Inequalities: theory of majorization and its applications},
  volume 143.
\newblock Academic Pr, 1979.

\bibitem{marshall2010inequalities}
A.W. Marshall, I.~Olkin, and B.C. Arnold.
\newblock {\em Inequalities: theory of majorization and its applications}.
\newblock Springer, 2010.

\bibitem{maxwell1998normal}
G.~Maxwell.
\newblock The normal subgroups of finite and affine {C}oxeter groups.
\newblock {\em Proceedings of the London Mathematical Society}, 76(2):359--382,
  1998.

\bibitem{niezgoda2002structure}
Marek Niezgoda.
\newblock On the structure of a class of {Eaton} triples.
\newblock In {\em Forum Mathematicum}, volume~14, pages 405--412. Walter de
  Gruyter \& Co., 2002.

\bibitem{niezgoda2001norm}
Marek Niezgoda and Tin-Yau Tam.
\newblock On the norm property of {$G(c)$}-radii and {E}aton triples.
\newblock {\em Linear Algebra and its Applications}, 336(1):119--130, 2001.

\bibitem{steerneman1990gmajorization}
A.G.M Steerneman.
\newblock {$G$}-majorization, group-induced cone orderings, and reflection
  groups.
\newblock {\em Linear Algebra and its Applications}, 127:107--119, 1990.

\end{thebibliography}

\end{document}